\documentclass[12pt]{amsart}
\usepackage{txfonts}
\usepackage{bbm}
\usepackage{mathrsfs}
\usepackage{amscd}
\usepackage{amssymb}
\usepackage{amsmath, amsthm, graphics}
\usepackage[all]{xy}
\usepackage{cite}
\usepackage{tikz}
\usepackage{color}
\usepackage{tikz-cd}
\usepackage{hyperref}
\usepackage[english]{babel}
\usepackage[margin=1.3in]{geometry}

\newtheorem{theorem}{Theorem}[section]
\newtheorem{corollary}[theorem]{Corollary}
\newtheorem{lemma}[theorem]{Lemma}
\newtheorem{example}[theorem]{Example}
\newtheorem{proposition}[theorem]{Proposition}

\newtheorem{remark}[theorem]{Remark}

\begin{document}

\title{Quasi-modularity  and holomorphic anomaly equation for the twisted Gromov-Witten theory: $\mathcal{O}(3)$ over $\mathbb{P}^2$}
\author{Xin Wang}
\address{Department of Mathematics\\ Shandong University \\ Jinan, China}

\email{xinwmath@gmail.com}
\begin{abstract}
In this paper, we prove quasi-modularity property for the twisted Gromov-Witten  theory of  $\mathcal{O}(3)$ over $\mathbb{P}^2$. Meanwhile, we derive its holomorphic anomaly equation.
\end{abstract}
\maketitle
\tableofcontents
\allowdisplaybreaks

\section{Introduction}
Let $X$ be a smooth compact Calabi-Yau threefold, denote
\[N_{g,d}:=\int_{[\overline{M}_{g,n}(X,d)]^{\text{vir}}}1\]
to be the genus-$g$, degree $d\in H_{2}(X;\mathbb{Z})$ Gromov-Witten invariants, where $[\overline{M}_{g,n}(X,d)]^{\text{vir}}$ is the virtual fundamental class of the moduli space of stable map $\overline{M}_{g,n}(X,d)$ (c.f. \cite{li1998virtual} and \cite{ruan1997higher}).
Define \[F_{g}:=\sum_{d=0}^{\infty}N_{g,d}q^d\]
 to be the genus-$g$ generating funtions. Then there are two predictions from mirror symmetry:
\begin{itemize}
\item{There is a finitely generated subring of ``quasi-modular objects"
\[\mathcal{R}\in \mathbb{Q}[[q]],\] such that for all $g$
\[F_g\in \mathcal{R}.\]
}
\item{$F_g$ satisfies a recursion formula, which is called holomorphic anomaly equation. i.e.
    \[\frac{\partial}{\partial \bar{q}}F_g(q,\bar{q})=\text{exact formula of}F_{h<g}.\]
    where $F_g(q,\bar{q})$ is the modular completion of $F_g(q)$.}
\end{itemize}

In  general, it is usually very difficult to compute the Gromov-Witten invariants for compact Calabi-Yau threefold. Recently, it is proved that these two properties are satisfied for Quintic 3-fold (Ref. \cite{Guo2019}).  Instead, the twisted Gromov-Witten theory is much easier and also expected to satisfy the two properties in the above predictions. Many examples such as local $\mathbb{P}^2$ has been studied (Ref. \cite{lho2018stable}, \cite{coates2018gromov} and \cite{fang2018open}).
In this paper, inspired by the calculations in \cite{Guo2019}, we discuss one simple example $\mathcal{O}(3)$ over $\mathbb{P}^2$  and prove the quasi-modularity property and holomorphic anomaly equation.

The reason why we are interested in this twisted theory $\mathcal{O}(3)$ over $\mathbb{P}^2$ is: it has close relation with geometry of cubic elliptic curve, since the zeros of the generic section defines a cubic elliptic curve.

Our main theorem is
\begin{theorem}\label{theoremelliptic}
The formal twisted  theory $\mathcal{O}(3)$ over $\mathbb{P}^2$ satisfies
\begin{itemize}
\item{Quasi-modularity  property: the complex degree $g-1+n$ part  $[{\Omega^{\tau(q)}}_{g,n}(H,...,H)]_{g-1+n}$ is a quasi modular form of modular  group $\Gamma_0(3)$ with weight $2g-2+2n$ with valued in $A^{g-1+n}(\overline{\mathcal{M}}_{g,n};\mathbb{Q})$.
}
\item{The holomorphic anomaly equation:
\begin{align*}
&\frac{\partial}{\partial E_{2}}{\Omega^{\tau(q)}}_{g,n}(H,...,H)
\\=&\frac{1}{12}\sum_{j=1}^{n}\psi_{j}
{\Omega^{\tau(q)}}_{g,n}(H,...,H,1,H,...,
H)
\\&-\frac{1}{36}
i_{*}\left({\Omega^{\tau(q)}}_{g-1,n+2}(H,...,H,
1,1)
\right)
\\&-\frac{1}{36}
j_{*}\sum_{\substack{g_1+g_2=g\\S_1\sqcup S_2=\{1,...,n\}}}{\Omega^{\tau(q)}}_{g_1,|S_1|+1}
(H^{S_1},1)\otimes
{\Omega^{\tau(q)}}_{g_2,|S_2|+1}(H^{S_2},1)
\end{align*}
}
\end{itemize}
where $\tau(q)$ is the mirror map (see section~\ref{sec:O(3)}). $E_2$ is the second Eisenstein series. $i$ and $j$ are the gluing maps between moduli space of curves.
\[i: \overline{\mathcal{M}}_{g-1,n+2}\rightarrow \overline{\mathcal{M}}_{g,n}, \quad j: \overline{\mathcal{M}}_{g_1,|S_1|+1}\times \overline{\mathcal{M}}_{g_2,|S_2|+1}\rightarrow \overline{\mathcal{M}}_{g,n}.\]
\end{theorem}

\begin{remark}
Our theorem is somehow parallel to the quasi modularity and holomorphic anomaly equation of elliptic curve proved in \cite{oberdieck2018holomorphic}.
The main difference of the quasi modularity of twisted theory $\mathcal{O}(3)$ over $\mathbb{P}^2$ and elliptic curve is about the modular ring: one is $\Gamma_0(3)$, the other is $SL(2;\mathbb{Z})$.  The associated holomorphic anomaly equation is the same shape.
\end{remark}
This paper is organised as follows: In section~\ref{sec:pre}, we review some basic knowledge in twisted Gromov-Witten theory. In section~\ref{sec:O(3)}, we focus on the genus-0 twisted Gromov-Witten theory of the example $\mathcal{O}(3)$ over $\mathbb{P}^2$ and the computation of $R$ matrix. In section~\ref{sec:fg} and section~\ref{sec:modular}, we prove the finite generation  and quasi modularity property of twisted Gromov-Witten theory of the example $\mathcal{O}(3)$ over $\mathbb{P}^2$. In section~\ref{sec:HAE}, we prove the associated holomorphic anomaly equation.
\\{\bf  Acknowledgements.}
The authors would like to special thank Shuai Guo and Felix Janda for discussing Givental theory and Calabi-Yau geometry. The results are obtained during the visit of the author in University of Michigan. The author would like thank professor Melissa Liu and Yongbin Ruan for their help during the visit in Columbia University and University of Michigan.
The author is partially supported by  NSFC grant 11601279.

\section{Preliminary on twisted Gromov-Witten theory}\label{sec:pre}
Let $X$ be a smooth projective variety, $E$ is a holomorphic vector bundle over $X$. $\overline{M}_{g,n}(X,\beta)$ is the moduli space of stable maps to $X$ of genus-$g$, with $n$ markings, degree $\beta$. We consider the universal stable maps
over $\overline{M}_{g,n}(X,\beta)$
\[\xymatrix{
  \overline{M}_{g+1,n}(X,\beta) \ar[d]_{\pi} \ar[r]^{f} & X       \\
  \overline{M}_{g,n}(X,\beta)                     }\]
Set the K-theoretical push forward
\[E_{g,n,\beta}=R^0\pi_*f^*E-R^1\pi_*f^*E \in K^0(\overline{M}_{g,n}(X,\beta))\]
Let $\mathbf{c}$ be an invertible multiplicative character class.
For any $v_1,v_2\in H^*(X;\mathbb{Q})$, the twisted paring is defined by
\[(v_1,v_2)^{tw}:=\int_{X}v_1v_2\mathbf{c}(E)\]
and the correlator is defined by
\[\Omega_{g,n}(v_1,...,v_n):=\sum_{\beta=0}^{\infty}q^{\beta}p_{*}\left([\overline{M}_{g,n}(X,\beta)]^{vir}
ev_1^*v_1...ev_n^*v_n\mathbf{c}(E_{g,n,\beta})\right)\in H^*(\overline{M}_{g,n})[[q]]\]
Then the Cohomological field theory $(H^*(X,\mathbb{Q}),(,)^{tw},\Omega,1)$ with unit $1\in H^*(X,\mathbb{Q})$ is our twisted theory.
Similarly, we can define the (decedant) twisted genus-$g$ GW invariants by
\[\langle {\psi_1}^{k_1}v_1,...,{\psi_n}^{k_n}v_n\rangle_{g,n,\beta}
=\int_{\overline{M}_{g,n}}{\psi_1}^{k_1}ev_1^*v_1...{\psi_n}^{k_n}ev_n^*v_n
\mathbf{c}(E_{g,n,\beta}).\]
In general, there are many different type of twisted theory associated to a holomorphic vector bundle $E$ over a smooth projective variety $X$.
In this paper,  we study a certain type of  twisted theory, called formal theory (c.f. \cite{lho2018stable}). As an example, we give the formal theory associated to  $\mathcal{O}(3)$ over $\mathbb{P}^2$.
\begin{example}
Consider the $(\mathbb{C}^*)^3$  action on the base $\mathbb{P}^2$, with action
\[(t_0,t_1,t_2)\cdot[x_0,x_1,x_2]=[{t_0}^{-1}x_0,{t_1}^{-1}x_1,{t_2}^{-1}x_2]\]
then this torus action has a natural lift to the total space $\mathcal{O}_{\mathbb{P}^2}(3)$ by
\[(t_0,t_1,t_2)\cdot[x_0,x_1,x_2,v]=[{t_0}^{-1}x_0,{t_1}^{-1}x_1,{t_2}^{-1}x_2,v]\]
denote the associated equivariant parameter of the torus action $\lambda_0,\lambda_1,\lambda_2$, then taking the specialization $\lambda_i=\xi^{i}\lambda$, where $\xi=e^{\frac{2\pi i}{3}}$.
The invertible multiplier character class is the equivariant Euler class
$\mathbf{e}_{(\mathbb{C}^*)^3}(\mathcal{O}_{\mathbb{P}^2}(3))$ (taking specialization).  Then we obtain the associated formal twisted theory.
\end{example}
\subsection{Genus-0 twisted Gromov-Witten theory}
The genus-0 twisted GW invariants was introduced and studied in \cite{coates2007quantum} and \cite{coates2009computing}.  There are two important ingredients in genus-0 twisted theory: one is the twisted Lagrange cone, the other one is the quantum differential equation.
\subsubsection{Twisted Lagarian cone}
Assume $\{\phi_{\alpha}\}_{\alpha=1}^{N}$ is a linear basis of $H^*(X,\mathbb{Q})$.
A generic element in the Lagrangian cone $\mathcal{L}^{tw}$ has the form
\[J^{tw}(t(z),-z)=-1z+t(z)+\sum_{n,d}\frac{q^d}{n!}\langle t(\psi),...,t(\psi),\frac{\phi^\alpha}{-z-\phi_{n+1}}\rangle_{0,n+1,d}^{tw}\phi_{\alpha}\]
where $t(z)\in H^*(X,\mathbb{Q})[[z]]$.
By string equation, dilaton equation and genus-0 topological recursion relation, Givental proved the Lagrangian cone has very special geometry (c.f. \cite{givental2004symplectic}):
\begin{itemize}
\item{$\mathcal{L}^{tw}$ is a cone at the origin}
\item{Each tangent space $T$ is tangent to  the cone $\mathcal{L}^{tw}$ exactly along $zT$}
\end{itemize}
This implies that
\begin{itemize}
\item{The cone $\mathcal{L}^{tw}$ is determined by big $J$ function: $J(t,-z)$}
\item{Each tangent space $T$ is a $\Lambda[[z]]$ module with basis $S^{\text{tw}}(t,-z)^*\phi_{1},...,S^{\text{tw}}(t,-z)^*\phi_n$.}
\item{\[\mathcal{L}^{tw}=\left\{z\sum_{i=1}^{n}c_i(t,z)S^{\text{tw}}(t,-z)^*\phi_{i}\right\}\]}
\end{itemize}
where $S$ operator is defined by 2-pt function
\[{S^{\text{tw}}}(t,z)^*(v)=v+\sum_{n,d,\alpha}\frac{q^d}{n!}\langle v,t,...,t,\frac{\phi_\alpha}{z-\psi}\rangle_{0,n+2,d}^{\text{tw}}\phi^\alpha\]
\subsubsection{Quantum differential equation}
Quantum differential equation  at point $t\in H^*(X;\mathbb{Q})$ has the form
\begin{align}\label{QDE}
zdS(t,q)=dt*_{t}S(t,q).
\end{align}

By genus-0 topological recursion relation, the two point function $S(t,q)$ is always a fundamental solution of the QDE \eqref{QDE}.
Near the semisimple point $t=\tau$, there is a formal solution of the form $\Psi R e^{\frac{U}{z}}$, where $R$ is unique up to some constant matrix of power series of $z$: $e^{\sum_{k=0}^{\infty}z^{2k+1}diag(a^{(k)}_{1},...,a^{(k)}_{n})}$, where all $a_{j}^{(i)}\in \mathbb{C}$ (c.f. \cite{givental2001semisimple}).
\begin{remark}
In general, there is no relation between these two solutions. In some good cases, for example, the full equi-variant theory of a smooth toric variety. These 2 solution can be related by localization and quantum Riemann-Roch. The ambiguity of the $R$ matrix can be uniquely determined by quantum Riemann-Roch (c.f. \cite{givental2001semisimple}.).
\end{remark}
\subsection{Higher genus twisted Gromov-Witten invariants.}
Most examples of twisted theory we are interested are semi-simple. The semisimplicity of our examples can be seen from the proof.  So we can directly apply Givental-Teleman classification theorem (\cite{teleman2012structure}) to compute higher genus twisted GW invariants from genus-0 twisted GW invariants. Formally, the formula is
(for precise formula, see\cite{pandharipande2015relations})
\[\Omega=R_{\cdot}\omega\]
where $\omega_{g,n}=\Omega_{g,n}\cap H^0(\overline{M}_{g,n})$ is the topological part of $\Omega$.
It is a graph sum formula. All the contributions in the graph sum are expressed in terms of $R$ matrix.

So if we want to compute higher genus twisted GW invariants, the key point is to compute the $R$ matrix here.
Before compute it, we should first know where $R$ matrix comes from and its geometric meaning. The original proof in \cite{teleman2012structure} relies heavily on the topological structure of the moduli space of curves. For equivariant Gromov-Witten theory, there is a more geometric way to get the $R$ matrix. For example, for smooth  toric variety, first taking localization, then contract the unstable rational components to get stable graphs, in this process, we get a matrix $R_{loc}$, which comes from the contribution of all the unstable rational components. Then using quantum Riemann-Roch theorem, string and dilaton  equations, we get the $R$ marix.

\section{Twisted theory $\mathcal{O}(3)$ over $\mathbb{P}^2$}
\label{sec:O(3)}
In this section, we consider the formal twisted  theory associated to $\mathcal{O}(3)$ over $\mathbb{P}^{2}$.
\subsection{Classical $(\mathbb{C}^*)^3$ equivariant theory for $\mathbb{P}^{2}$}
In this subsection, we compute $(\mathbb{C}^*)^3$ equivariant theory for $\mathbb{P}^{2}$, where the action is given by
\[(t_0,t_1,t_2)\cdot[z_0,z_1,z_2]:=[{t_0}^{-1} z_0,{t_1}^{-1} z_1,{t_2}^{-1} z_2]\]
then the classical equivariant cohomology ring is
\[H^*_{(\mathbb{C}^*)^3}(\mathbb{P}^2)\cong \mathbb{C}[H]/\langle \prod_{i=0}^{2}(H-\lambda_i)=0\rangle\]
where $H=c_1^{\mathbb{C}^*}(\mathcal{O}(1))$ under the lift of the torus action to $\mathcal{O}(1)$, which satisfy $H|_{p_{\alpha}}=\lambda_\alpha$.
Choose a flat basis $\{1,H,H^2\}$ of $H^*_{(\mathbb{C}^*)^3}(\mathbb{P}^2)$.
The twisted paring by twisted bundle $\mathcal{O}(3)$ is: any $v_1,v_2$,
\[(v_1,v_2):=\int_{\mathbb{P}^2}v_1v_2e_{(\mathbb{C}^*)^3}(\mathcal{O}(3))
=\int_{\mathbb{P}^2}v_1v_2(3H)\]
From now on, for simplicity, we specialize the equivariant parameter by $\lambda_\alpha:=\xi^\alpha\lambda$.
Then the twisted paring matrix under basis $\{1,H,H^2\}$ is
\begin{align}\label{paring-matrix}
\begin{pmatrix}
0&3&0\\
3&0&0\\
0&0&3\lambda^3
\end{pmatrix}.
\end{align}
The algebra $H^*_{\mathbb{C}^*}(\mathbb{P}^2)$ is semisimple. It has a canonical basis
\[\mathfrak{e}_\alpha=\frac{\prod_{\beta\neq\alpha}(H-\xi^\beta\lambda)}{\prod_{\beta\neq\alpha}(\xi^\alpha\lambda-\xi^\beta\lambda)}
=\frac{\prod_{\beta\neq\alpha}(H-\xi^\beta\lambda)}{3\xi^{2\alpha}\lambda^2}\]
satisfying the properties of idempotent of the semisimple algebra:
\[\sum_{\alpha}\mathfrak{e}_\alpha=1,\quad \mathfrak{e}_\alpha\cdot\mathfrak{e}_\beta=\delta_{\alpha\beta}\mathfrak{e}_\alpha.\]
It is easy to compute \[H\cdot\mathfrak{e}_\alpha=(\xi^\alpha\lambda)\mathfrak{e}_\alpha=:\mathfrak{h}_{\alpha}\mathfrak{e}_\alpha\]
and its norm
\[\|\mathfrak{e}_\alpha\|=(1,\bar{\mathfrak{e}}_\alpha)^{\lambda}
=\sqrt{\frac{1}{\xi^{\alpha}\lambda}}.\]
\subsection{$I$ function and mirror theorem}
First, the correlator of the twisted  theory associated to $\mathcal{O}(3)$ over $\mathbb{P}^2$ now is defined to be
\[\Omega_{g,n}(v_1,...,v_n)=\sum_{d=0}^{\infty}q^dp_{*}\left(ev_1^*{v_1}...ev_n^*{v_n}
e_{(\mathbb{C}^*)^{3}}(R\pi_{*}f^*\mathcal{O}(3))\cap\left[\overline{M}_{g,n}(\mathbb{P}^2,d)\right]^{vir}\right)\]
In the following, we consider the associated shifted CohFT
\[(\Omega^{\tau(q)})_{g,n}(v_1,...,v_n)
=\sum_{k=0}^{\infty}\frac{1}{k!}(\pi_k)_*\left(\Omega_{g,n+k}(v_1,...,v_n,\tau(q)^{\otimes k})\right)
\]
with initial valuation
\begin{align*}
&(\Omega^{\tau(q)}|_{q=0})_{g,n}(v_1,...,v_n)
=(\Omega|_{q=0})_{g,n}(v_1,...,v_n)
\\=&p_{*}\left(ev_1^*{v_1}...ev_n^*{v_n}
e_{(\mathbb{C}^*)^{3}}(R\pi_{*}f^*\mathcal{O}(3))
\cap\left[\overline{M}_{g,n}(\mathbb{P}^2,0)\right]^{vir}\right)
\end{align*}
To compute $R$ matrix, we should use another family of elements $I(q,-z)$ in the twisted Lagrangian cone $\mathcal{L}^{\mathcal{O}(3)}_{\mathbb{P}^2}$, which is a hypergeometric series.
\begin{align*}
&I(q,z)=z\sum_{d=0}^{\infty}q^d\frac{\prod_{k=1}^{3d}(3H+kz)}{\prod_{k=1}^d\left((H+kz)^3-\lambda^3\right)}
=z\sum_{k\geq0}I_k(q)\left(\frac{H}{z}\right)^k
\end{align*}
with $I_{0}(q)=1+\sum_{d=1}^{\infty}q^d\frac{(3d)!}{(d!)^3}$,
$I_{1}(q)=\sum_{d=1}^{\infty}q^d\frac{(3d)!}{(d!)^3}\left(\sum_{m=d+1}^{3d}\frac{3}{m}\right)$.
A very important property of the $I$ function is: it is the solution of  Picard-Fuchs equation
\[\left(D_H^3-\lambda^3-q\prod_{k=1}^{3}(3D_{H}+kz)\right)I(q,z)=0\]
where $D_H=H+zq\frac{d}{dq}$. Take paring with normalized fixed point basis $\{\bar{\mathfrak{e}}_{\alpha}\}_{\alpha=1}^{3}$, we can replace $H$ with $\{\mathfrak{h}_\alpha\}$, then the Picard-Fuchs equation becomes
\[\left((zq\frac{d}{dq}+\mathfrak{h}_{\alpha})^3
-\lambda^3-q\prod_{k=1}^{3}(3zq\frac{d}{dq}+3\mathfrak{h}_{\alpha}+kz)\right)\left\langle I(q,z),\bar{\mathfrak{e}}_{\alpha}\right\rangle
=0.\]
\subsection{Computation of  $R(z)^*1$}
Via genus-0 mirror theorem (c.f. \cite{givental1996equivariant} and \cite{lian1997mirror}),
\[\frac{I(q,z)}{I_{0}(q)}=J(\tau(q),q,z)=zS(\tau(q),q,z)^*1\]
where  $\tau(q)=\frac{I_1(q)}{I_0(q)}$ is the mirror map.
By localization and quantum Riemann-Roch theorem (c.f. \cite{coates2007quantum}), we have the relation between $S$ operator and $R$ operator,
\[\langle S(\tau(q),q,z)^*1,\bar{\mathfrak{e}}_{\alpha}\rangle \Delta_{\alpha}^{\text{tw}}(z)^*
=\langle 1, R(\tau(q),q,z)\bar{e}_{\alpha}(\tau(q),q)\rangle e^{u^\alpha(\tau(q),q)/z}\]
where
\[\Delta_{\alpha}^{\text{tw}}(z)^*=e^{-\sum_{k>0}\frac{B_{2k}}{2k(2k-1)}z^{2k-1}
(\frac{1}{(-3\lambda_\alpha)^{2k-1}}+\sum_{\beta\neq\alpha}\frac{1}{(\lambda_{\alpha}-\lambda_{\beta})^{2k-1}})}\]
which comes from contribution of quantum Riemann-Roch theorem.
Thus $I_{0}R_{0\bar{\alpha}}(\tau(q),q,z)$ satisfies Picard-Fuchs equation
\begin{align}\label{PFforR}
\left((zq\frac{d}{dq}+L_{\alpha})^3
-\lambda^3-q\prod_{k=1}^{3}(3zq\frac{d}{dq}+3L_{\alpha}+kz)\right)
\left(I_{0}R_{0\bar{\alpha}}(\tau(q),q,z)\right)
=0
\end{align}
with initial condition
i.e.
\begin{align*}
I_{0}(q)\Psi_{0\bar{\alpha}}(\tau(q),q)|_{q=0}=\langle 1,\bar{\mathfrak{e}}_\alpha \rangle,
\quad {R_{0}}^{\alpha}(\tau(q),q,z)|_{q=0}
=\Delta_{\alpha}^{\text{tw}}(z)^*,
\end{align*}
where $L_{\alpha}(q)=\mathfrak{h}_{\alpha}+q\frac{d}{dq}u^{\alpha}(\tau(q),q)$.
By looking at coefficient of $z^0$ in the equation~\eqref{PFforR}.
\begin{align}\label{norm}
I_{0}(q)\Psi_{0\bar{\alpha}}(q)=\frac{L_\alpha}{\sqrt{\lambda^3}}
=(\xi^\alpha\lambda)^{-\frac{1}{2}}L
\end{align}
where $L:=(1-27q)^{-\frac{1}{3}}$.
Then by looking at the coefficients of $z^{k+1}, k\geq0$, we can get all the ${(R_k)_{0}}^{\alpha}$ recursively by solving the Picard-Fuchs equation and initial conditions. For example, here we list  ${(R_k)_{0}}^{\alpha}$ for $k=1,2,3$.
\begin{align*}
&{(R_1)_{0}}^{\alpha}=\frac{-L_\alpha^2}{18\lambda^3}=\frac{-\xi^{2\alpha}L^2}{18\lambda}
\\&{(R_2)_{0}}^{\alpha}=\frac{L_\alpha^4}{648\lambda^6}=\frac{\xi^\alpha L^4}{648\lambda^2}
\\&{(R_3)_{0}}^{\alpha}=\frac{2875L_\alpha^6-3600L_\alpha^3\lambda^3+702\lambda^6}{174960\lambda^9}
=\frac{2875L^6-3600L^3+702}{174960\lambda^3}
\end{align*}

In general, we have the following property of $R(z)^*1$
\begin{proposition}\label{Rstar1}
\begin{align*}
{(R_k)_{0}}^{\alpha}\in (\xi^\alpha\lambda)^{-k}\cdot \mathbb{Q}\left\{L^{2k-3j}:0\leq j\leq \lfloor\frac{2}{3}k\rfloor\right\}.
\end{align*}
\end{proposition}
\begin{proof}
For simplicity, we write
\[I_{0}(q)R_{0\bar{\alpha}}(q)=\frac{L_\alpha}{\sqrt{\lambda^3}}\cdot
(1+\frac{r_1(q)}{\xi^\alpha\lambda}z+\frac{r_2(q)}{(\xi^\alpha\lambda)^2}z^2+...)\]
then the Picard-Fuchs equation becomes
\begin{align}\label{PFr}\sum_{k=1}^{3}\frac{z^k}{(\xi^\alpha\lambda)^k L^k}\mathcal{D}_{k}(1+\frac{r_1(q)}{\xi^\alpha\lambda}z+\frac{r_2(q)}{(\xi^\alpha\lambda)^2}z^2+...)
=0,
\end{align}
with initial condition
\[(1+\frac{r_1(q)}{\xi^\alpha\lambda}z+\frac{r_2(q)}{(\xi^\alpha\lambda)^2}z^2+...)|_{q=0}
=e^{-\sum_{k>0}\frac{B_{2k}}{2k(2k-1)}z^{2k-1}
(\frac{1}{(-3\lambda_\alpha)^{2k-1}}+\sum_{\beta\neq\alpha}\frac{1}{(\lambda_{\alpha}-\lambda_{\beta})^{2k-1}})}\]
where
\begin{align*}
&\mathcal{D}_{1}=3D
\\&\mathcal{D}_{2}=3D^{(2)}-(L^3-1)D+\frac{1}{9}L^3(L^3-1)
\\&\mathcal{D}_{3}=D^{(3)}-(L^3-1)D^{(2)}-\frac{2}{9}(L^3-1)D
+\frac{1}{27}L^3(4L^3-1)(L^3-1).
\end{align*}
Notice \[D=q\frac{d}{dq}=\frac{1}{3}L(L^3-1)\frac{d}{dL}.\]
Observing that the coefficient of $z^{k+1}$ in equation~\eqref{PFr} gives us
\begin{align*}
&D(r_1)=-\frac{1}{27}L^2(L^3-1)\\
&D(r_2)=\frac{1}{81L}\left((27L^3-27)D(r_1)
-4L^8-3L^6r_1+5L^5+3L^3r_1-L^2-81D^{(2)}(r_1)\right)\\
&D(r_k)=\frac{1}{81L^2}\Big((27L^3-27)D^{(2)}(r_{k-2})
+(6L^3-6)D(r_{k-2})+(27L^4-27L)D(r_{k-1})-3L^4(L^3-1)r_{k-1}
\\&\hspace{70pt}-L^3(4L^3-1)(L^3-1)r_{k-2}
-81LD^{(2)}(r_{k-1})-27D^{(3)}(r_{k-2})\Big)
\end{align*}
for any $k\geq3$.

Then we solve the equation~\eqref{PFr} by induction, assume any $i<k$,
\[r_i\in\mathbb{Q}\left\{L^{2i-3j}:0\leq j\leq \lfloor\frac{2}{3}i\rfloor\right\}.\]
then
\[D(r_k)\in(L^3-1)\cdot\mathbb{Q}\left\{L^{2k-3j}:0\leq j\leq\lfloor\frac{2k-1}{3}\rfloor\right\}\]
After integration, we have
\[r_k\in\mathbb{Q}\left\{L^{2k-3j}:0\leq j\leq\lfloor\frac{2k-1}{3}\rfloor\right\}+\text{constant}\]
Then the proposition~\ref{Rstar1} follows from the following initial condition which will be proved by proposition~\ref{pro:osciasym} and lemma~\ref{lemma:osciinitial} via oscillatory integral and its asymptotic expansion.
\begin{align*}
&\lim_{L\rightarrow 0}r_k=0, \hspace{40pt} \text{if }\quad k\neq0(\text{mod}3)
\\&\lim_{L\rightarrow 0}r_k=\text{constant}, \quad \text{if }\quad k=0(\text{mod}3).
\end{align*}
\end{proof}
\subsection{Oscillatory integral and asymptotic expansion} \footnote{The argument in this section is similar to the Appendix A in \cite{lho2018crepant}.}
\label{asymptoic}
Now we consider the Landau-Ginzburg potential $W: (\mathbb{C}^*)^4\rightarrow\mathbb{C}$ for the $(\mathbb{C}^*)^3$ equivariant twisted theory $\mathcal{O}(3)$ over $\mathbb{P}^2$,
\[W(x_0,x_1,x_2)=\sum_{i=0}^{2}(x_i+\lambda_i\ln x_i)+x_3,\quad x_0x_1x_2 (-x_3)^{-3}=q.\]
%
Now we consider the oscillatory integral
\begin{align*}
&\mathcal{I}(q,z,\lambda)=\int_{\gamma}e^{\frac{W}{z}}q\frac{dx_0dx_1dx_2dx_3}{x_0x_1x_2 x_3dq}
\end{align*}
where $\gamma$ is the Lefschetz thimble in a sub-torus in $(\mathbb{C}^*)^4$
\[\gamma\subset\{x_0x_1x_2(-x_3)^{-3}=q\}\subset(\mathbb{C}^*)^4\]
We consider the following 3 local charts of the sub-torus
\[U_0=\{(\frac{-q}{x_1x_2{x_3}^{-3}},x_1,x_2,x_3)\}\]
\[U_1=\{(x_0,\frac{-q}{x_0x_2{x_3}^{-3}},x_2,x_3)\}\]
\[U_2=\{(x_0,x_1,\frac{-q}{x_0x_1{x_3}^{-3}},x_3)\}\]
then we have the following proposition, which gives the relation between the $R(z)^*1$ and oscillatory integral.
\begin{proposition}\label{pro:osciasym}
For $\alpha=0,1,2$, the derivative of $\mathcal{I}(q,z,\lambda)|_{U_\alpha}$ has the asymptotic
  expansion
  \begin{equation*}
    z q\frac{d}{dq} \mathcal{I}(q,z,\lambda)|_{U_\alpha}\cdot1z
      \asymp^{z\rightarrow0}   -\frac{1}{3}(2\pi  z)^{\frac{3}{2}}  e^{\frac{\check{u}_\alpha}{z}}\cdot\frac{L_\alpha}{\sqrt{\lambda^3}}  \cdot \langle 1, R(z)(e^\alpha)\rangle,
  \end{equation*}
with
\[q\frac{d}{dq}\check{u}_\alpha=L_\alpha=\xi^\alpha\lambda\cdot L.\]
\end{proposition}
\begin{proof}
For simplicity, we only prove for $\alpha=0$. In the local chart $U_0$, the oscillatory
integral $\mathcal{I}(q,z)$ becomes:
\begin{align*}
&\mathcal{I}(q,z)|_{U_0}
=\int_{\gamma\cap U_0}e^{\frac{W}{z}}d\ln x_1d\ln x_2d\ln x_3
\\=&\int_{\gamma\cap U_0}(-q)^{\frac{\lambda_0}{z}}e^{\frac{1}{z}
\left(x_1+x_2+x_3-(\lambda_0-\lambda_1)\ln x_1
-(\lambda_0-\lambda_2)\ln x_2
-(-3\lambda_0)\ln x_3
\right)}
\sum_{d\geq0}\frac{(-q)^d}{d!z^d}{x_1}^{-d}{x_2}^{-d}{x_3}^{3d}
\frac{dx_1dx_2dx_3}{x_1x_2x_3}
\\=&(-q)^{\frac{\lambda_0}{z}}
\sum_{d\geq0}\frac{(-q)^d}{d!z^d}
\int_{0}^{\infty}e^{\frac{x_1}{z}}{x_1}^{\frac{-(\lambda_0-\lambda_1)}{z}-d}\frac{dx_1}{x_1}
\int_{0}^{\infty}e^{\frac{x_2}{z}}{x_2}^{\frac{-(\lambda_0-\lambda_2)}{z}-d}\frac{dx_2}{x_2}
\int_{0}^{\infty}e^{\frac{x_3}{z}}{x_3}^{\frac{-(-3\lambda_0)}{z}+3d}\frac{dx_3}{x_3}
\\=&(-q)^{\frac{\lambda_0}{z}}
\sum_{d=0}^{\infty}\frac{(-q)^d}{d!z^d}
\frac{\prod_{k=0}^{3d-1}(-3\lambda_0-kz)}{\prod_{k=1}^d
\left(
(kz+\lambda_0-\lambda_1)
(kz+\lambda_0-\lambda_2)\right)}
\\&\hspace{40pt}\cdot\int_{0}^{\infty}e^{\frac{x_1}{z}}{x_1}^{\frac{-(\lambda_0-\lambda_1)}{z}}\frac{dx_1}{x_1}
\int_{0}^{\infty}e^{\frac{x_2}{z}}{x_2}^{\frac{-(\lambda_0-\lambda_2)}{z}}\frac{dx_2}{x_2}
\int_{0}^{\infty}e^{\frac{x_3}{z}}{x_3}^{\frac{-(-3\lambda_0)}{z}}\frac{dx_3}{x_3}
\end{align*}
In the above, we use integration by parts repeatedly  and the vanishing of the integration equals to zero at the ends of the Lefchetz thimble.

Notice that
\begin{align*}
I(q,z)|_{p_0}=z\sum_{d=0}^{\infty}\frac{q^d}{d!z^d}
\frac{\prod_{k=1}^{3d}(3\lambda_0+kz)}{\prod_{k=1}^d
\left(
(kz+\lambda_0-\lambda_1)
(kz+\lambda_0-\lambda_2)\right)}
\end{align*}
Then we have
\begin{align*}
&3zq\frac{d}{dq} \mathcal{I}(q,z)|_{U_0}\cdot 1z
\\=&(-q)^{\frac{\lambda_0}{z}}(3H)
I^{\text{tw}}(q,z)|_{p_0}
\\&\int_{0}^{\infty}e^{\frac{x_1}{z}}{x_1}^{\frac{-(\lambda_0-\lambda_1)}{z}}\frac{dx_1}{x_1}
\int_{0}^{\infty}e^{\frac{x_2}{z}}{x_2}^{\frac{-(\lambda_0-\lambda_2)}{z}}\frac{dx_2}{x_2}
\int_{0}^{\infty}e^{\frac{x_3}{z}}{x_3}^{\frac{-(-3\lambda_0)}{z}}\frac{dx_3}{x_3}
\\ \asymp&^{z\rightarrow0_{-}}(-q)^{\frac{\lambda_0}{z}}(3\lambda_0)
I^{\text{tw}}(q,z)|_{p_0}
(-z)^{\frac{\lambda_0-\lambda_1}{-z}}\Gamma(\frac{\lambda_0-\lambda_1}{-z})
(-z)^{\frac{\lambda_0-\lambda_2}{-z}}\Gamma(\frac{\lambda_0-\lambda_2}{-z})
(-z)^{\frac{-3\lambda_0}{-z}}\Gamma(\frac{-3\lambda_0}{-z})
\\&(\lambda_0-\lambda_1)^{-\frac{1}{2}}
(\lambda_0-\lambda_2)^{-\frac{1}{2}}
(-3\lambda_0)^{-\frac{1}{2}}
e^{-\sum_{m\geq1}\frac{B_{2m}}{2m(2m-1)}
\left(\left(\frac{z}{\lambda_0-\lambda_1}\right)^{2m-1}
+\left(\frac{z}{\lambda_0-\lambda_2}\right)^{2m-1}
+\left(\frac{z}{-3\lambda_0}\right)^{2m-1}
\right)}
\\ \asymp&^{z\rightarrow0_{-}}
(-1)^{-\frac{1}{2}}(-q)^{\frac{\lambda_0}{z}}(-2\pi z)^{\frac{3}{2}}
e^{\frac{\lambda_0-\lambda_1}{-z}\left(1+\ln(\lambda_0-\lambda_1)\right)}
e^{\frac{\lambda_0-\lambda_2}{-z}\left(1+\ln(\lambda_0-\lambda_2)\right)}
e^{\frac{-3\lambda_0}{-z}\left(1+\ln(-3\lambda_0)\right)}
\Delta_{p_0}^{\text{tw}}(z)^*I(q,z)|_{p_0}
\\=&
-(-q)^{\frac{\lambda_0}{z}}(2\pi z)^{\frac{3}{2}}
e^{\frac{\lambda_0-\lambda_1}{-z}\left(1+\ln(\lambda_0-\lambda_1)\right)}
e^{\frac{\lambda_0-\lambda_2}{-z}\left(1+\ln(\lambda_0-\lambda_2)\right)}
e^{\frac{-3\lambda_0}{-z}\left(1+\ln(-3\lambda_0)\right)}
I_0(q)e^{\frac{u_\alpha}{z}}\langle 1,R(q,z)(\bar{e}_0)\rangle
\\=&
-(2\pi z)^{\frac{3}{2}}I_0(q)
e^{\frac{\check{u}_\alpha}{z}}\langle 1,R(q,z)(\bar{e}_0)\rangle
\end{align*}
Here we use  the standard asymptotic expansion of Gamma function:
\begin{align}\label{eq:gamma-asym}
\ln\Gamma(x+t) \asymp (x+t-\frac{1}{2})\ln x-x+\ln\sqrt{2\pi}
+\sum_{k=1}^{\infty}\frac{(-1)^{k+1}B_{k+1}(t)}{k(k+1)}x^{-k}
\end{align}
in complex domain as $x\rightarrow \infty$ uniformly on $|arg x|\leq\pi-\epsilon$, $\epsilon$ is given in advance.

\end{proof}
\begin{lemma}\label{lemma:osciinitial}
  \begin{align*}
  \left({L_\alpha}^{-1}\cdot z q \frac{d}{dq} \mathcal{I}_\alpha\right)|_{L=0}
 \asymp^{z\rightarrow0_{-}}
-\frac{1}{3}(2\pi z)^{\frac{3}{2}}\frac{1}{\sqrt{\lambda^3}}
e^{\sum_{i}\frac{\lambda_i}{z}\left(-1+\ln(-\lambda_i)\right)}e^{\sum_{k\geq1}
\frac{(-1)^{k+1}B_{3k+1}(\frac{1}{3})}{k(3k+1)}(\frac{z}{\lambda})^{3k}}
  \end{align*}
\end{lemma}
\begin{proof}
Now we compute the oscillatory integral using coordinates $x_0,x_1,x_2$.
\begin{align*}
&{L_\alpha}^{-1}\cdot zq\frac{d}{dq}\mathcal{I}(q,z)
=\frac{1}{9}\int_{\Gamma_\alpha\subset(\mathbb{C}^*)^3}
\left(\frac{x_0x_1x_2}{-q {L_\alpha}^3}\right)^{\frac{1}{3}}e^{\frac{W}{z}}\frac{dx_0\wedge dx_1\wedge dx_2}{x_0x_1x_2}
\end{align*}
Notice that
$\left(-q{L_\alpha}^3\right)|_{L=0}=\frac{\lambda^3}{27},$
so at $L=0$, we have
  \begin{align*}
 & \left({L_\alpha}^{-1}\cdot z q \frac{d}{dq} \mathcal{I}\right)|_{L=0}
=\frac{1}{3}(\lambda^3)^{-\frac{1}{3}}
    \int_{\gamma_\alpha \subset (C^*)^3}   (x_0x_1x_2)^{\frac{1}{3}} e^{\sum_{i=0}^{2}  z^{-1} (x_i {+} \lambda_i \ln x_i)}  \frac{d x_0 d x_1 d x_{2}}{x_0 x_1  x_{2}}
\\ \asymp &^{z\rightarrow0_{-}}
\frac{1}{3}(\lambda^3)^{-\frac{1}{3}}\prod_{i=0}^{2}\left(
(-z)^{\frac{\lambda_i}{z}+\frac{1}{3}}\int_{\mathbb{R}_+} e^{-x_i+ \frac{\lambda_i}{z} \ln x_i}  \frac{d x_i}{{x_i}^{\frac{2}{3}}}
\right)
=\frac{1}{3}(\lambda^3)^{-\frac{1}{3}}\prod_{i=0}^{2}\left(
(-z)^{\frac{\lambda_i}{z}+\frac{1}{3}}\Gamma\left(\frac{\lambda_i}{z}+\frac{1}{3}\right)
\right)
\\ \asymp&^{z\rightarrow0_{-}}
-\frac{1}{3}(2\pi z)^{\frac{3}{2}}\frac{1}{\sqrt{\lambda^3}}
e^{\sum_{i}\frac{\lambda_i}{z}\left(-1+\ln(-\lambda_i)\right)}e^{\sum_{k\geq1}
\frac{(-1)^{k+1}B_{3k+1}(\frac{1}{3})}{k(3k+1)}(\frac{z}{\lambda})^{3k}}
\end{align*}
In the second step in the above, we use the replacement $\frac{x_i}{-z}\rightarrow x_i$. The last step follows from  the asymptotic expansion of Gamma function \eqref{eq:gamma-asym}.
\end{proof}
\subsection{Quantum differential equation}
The  quantum differential equation has the form $dS(t,q)=dt*_{t}S(t,q)$, where $d$ only action on $t$, not on $q$. For our use, we write it in terms of $S(t,z)^*$.
At point $\tau=\tau(q)$, the quantum differential equation becomes
\[D_{H}S(\tau(q),z)^*(1,H,H^2)=S(\tau(q),z)^*\left(\dot\tau*_{\tau(q)}(1,H,H^2)\right)\]
where $D_{H}=zq\frac{d}{dq}+H$ and $\dot\tau=H+q\frac{d\tau(q)}{dq}=I_{1,1}H$.
Assume
\[\dot\tau*_{\tau(q)}(1,H,H^2)
=(H+q\frac{d\tau(q)}{dq})*_{\tau(q)}(1,H,H^2)=(1,H,H^2)A.\]
Via Birkhoff factorization, we obtain
\begin{lemma} The quantum differential matrix
\begin{align*}
A=
\begin{pmatrix}
0&0&\lambda^3I_{3,3}\\
I_{1,1}&0&0\\
0&I_{2,2}&0\\
\end{pmatrix}
\end{align*}
where for any $n\geq m\geq1$
\[I_{m,n}:=\frac{I_{m-1,n-1}}{I_{m-1,m-1}}+q\frac{d}{dq}\frac{I_{m-1,n}}{I_{m-1,m-1}},\quad
I_{0,n}:=I_{n}.\]
\end{lemma}
\begin{proof}
Via Genus-0 mirror theorem (c.f.\cite{givental1996equivariant} and \cite{lian1997mirror}), we have
\begin{align*}
S(\tau(q),z)^*1=\frac{I(q)}{zI_{0}(q)}=1+\frac{I_1(q)}{I_0(q)}\frac{H}{z}+\frac{I_2(q)}{I_0(q)}\frac{H^2}{z^2}
+\frac{I_3(q)}{I_0(q)}\frac{H^3}{z^3}+O(z^{-4})
\end{align*}
Taking derivative along $D_H$
\begin{align*}
D_H(S(\tau(q),z)^*1)
=I_{1,1}H
+I_{1,2}z^{-1}H^2
+I_{1,3}z^{-2}H^3
+O(z^{-3})
=I_{1,1}(q)(S(\tau(q),z)^*H)
\end{align*}
So we get
\begin{align*}
S(\tau(q),z)^*H
=H+z^{-1}H^2\left(\frac{I_{1,2}}{I_{1,1}}\right)
+z^{-2}H^3\left(\frac{I_{1,3}}{I_{1,1}}\right)
+O(z^{-3})
\end{align*}
then taking derivatives along $D_H$ recursively, we obtain the
\[D_{H}\left(S(\tau(q),z)^*(1,H,H^2)\right)=S(\tau(q),z)^*(1,H,H^2)\begin{pmatrix}
0&0&\lambda^3I_{3,3}\\
I_{1,1}&0&0\\
0&I_{2,2}&0\\
\end{pmatrix}\]

\end{proof}
From symmetry of the quantum produc $\dot\tau*_{\tau}$ and character polynomial of matrix $A$, we  get following relations among the entries of matrix $A$.
\begin{lemma}\label{ZZYY}
Let $L(q):=(1-27q)^{-\frac{1}{3}}$, then
\begin{align}
&(i) \quad I_{1,1}I_{2,2}I_{3,3}=L^3, \quad I_{0,0}=I_{2,2}=I_{3,3}\nonumber
\\&(ii)\quad (1-27q)\left((q\frac{d}{dq})^2I_{0}\right)-27q\left(q\frac{d}{dq}I_{0}\right)-6qI_{0}=0
\nonumber\end{align}
\end{lemma}
\begin{proof}
The equations $(i)$  are exactly the Zinger-Zagier relation (c.f. \cite{zagier2007some}). Now we give the proof of equation $(ii)$.
Let
\begin{align*}
&\widetilde{I}(q,z)=\sum_{d=0}^{\infty}\frac{q^d}{H+dz}\frac{\prod_{k=1}^{3d}(3H+kz)}{\prod_{k=1}^d\left((H+kz)^3-\lambda^3\right)}
=3\sum_{d=0}^{\infty}q^d\frac{\prod_{k=1}^{3d-1}(3H+kz)}{\prod_{k=1}^d\left((H+kz)^3-\lambda^3\right)}
\end{align*}
then
\[D_{H}\widetilde{I}(q,z)=\frac{I(q,z)}{z}\]
The $\widetilde{I}$ function satisfies PF equation
\[\left(D_H^3-\lambda^3-q\prod_{k=0}^{2}(3D_{H}+kz)\right)\widetilde{I}(q,z)=0\]
It is equivalent to
\begin{align}
\label{ItildePF}\left(D_H^2-3q\prod_{k=1}^{2}(3D_{H}+kz)\right)\frac{I(q,z)}{z}
=\lambda^3\widetilde{I}(q,z)
\end{align}
It is easy to compute the left hand side of the above equation is
\begin{align*}
&\left({D_H}^2-3q(3D_H+z)(3D_H+2z)\right)\frac{I(q,z)}{z}
\\=&(1-27q)\left(I_{0}I_{1,1}I_{2,2}S^*{H^2}+(zq\frac{d}{dq}(I_{0}I_{1,1}))S^*H
+(zq\frac{d}{dq}I_{0})I_{1,1}S^*H+((zq\frac{d}{dq})^2I_{0})S^*1\right)
\\&-27qz\left(I_{0}I_{1,1}S^*H+(zq\frac{d}{dq}I_{0})S^*1\right)
-6qz^2I_{0}S^*1
\end{align*}
Thus $\widetilde{I}(q,z)$ lies in the tangent space $T\mathcal{L}^{\text{tw}}$ of the twisted Lagrangian cone. By the expansion
\[\lambda^3\widetilde{I}(q,z)=\lambda^3H^{-1}+O(z^{-1})=H^{2}+O(z^{-1})\]
we have
\[\lambda^3\widetilde{I}(q,z)=S^*{H^2},\]
then comparing the coefficient of $S^*1$ in the equation \eqref{ItildePF}, we get equation $(ii)$.
\end{proof}

\section{Finite generation property}\label{sec:fg}
\subsection{Basic generators}\label{subsec:generator}
Define the following degree $k$ generator
\[X_k=\left(L^{-1}q\frac{d}{dq}\right)^k\ln\left(\frac{I_{00}}{L}\right),\quad
Y_{k}:=\left(L^{-1}q\frac{d}{dq}\right)^{k}\ln \left(q^{\frac{1}{3}}L\right).\]

\begin{lemma} For generators $\{X_k\}$ and $\{Y_k\}$, we have
\begin{align*}
&X_2=-{X_1}^2-\frac{1}{2}Y_2\\&
Y_{k}\in \mathbb{Q}[Y_{1},Y_{2},Y_{3}]_{deg=k}
\end{align*}
\end{lemma}
\begin{proof}The first equation just follows from $(ii)$ in Lemma\ref{ZZYY}. The second one follows easily by induction from the following identities
\begin{align*}
&L = 9{Y_1}^2-\frac{9}{2}Y_2\\&
L^2 = 3Y_1\\&
1 = 27{Y_1}^3-\frac{135}{2}Y_1 Y_2+\frac{27}{2}Y_3\\&
L^{-1}q\frac{d}{dq}L=\frac{1}{3}L^3-\frac{1}{3}.
\end{align*}

\end{proof}
\subsection{Total $R$ matrix}
Together the quantum differential equation \eqref{QDE} and the relation between $S$ operator and $R$ operator, we obtain the recursion relation of total $R$ matrix.
\begin{align}\label{QDER}\left(zq\frac{d}{dq}+L_\alpha(q)\right)
(R_{0\bar{\alpha}},R_{1\bar{\alpha}},R_{2\bar{\alpha}})
=(R_{0\bar{\alpha}},R_{1\bar{\alpha}},R_{2\bar{\alpha}})A
\end{align}
where $R_{i\bar{\alpha}}:=\langle H^i,R(\tau(q),z)(\bar{e}_\alpha(q))\rangle$.
From the recursion formula of $R$ matrix, we obtain the important property of $R$ matrix
\begin{proposition}For any $k\geq0$,
\begin{align*}
&{(R_k)_{0}}^{\alpha}\in (\xi^\alpha\lambda)^{-k}\cdot \mathbb{Q}[L]_{\text{deg}=k}
\\&{(R_k)_{1}}^{\alpha}\in (\xi^\alpha\lambda)^{-k+1}\frac{{I_{00}}^2}{L^2}\cdot \mathbb{Q}[X_1,L]_{\text{deg}=k}
\\&
{(R_k)_{2}}^{\alpha}\in (\xi^\alpha\lambda)^{-k+2}\frac{{I_{00}}}{L}\cdot \mathbb{Q}[L]_{\text{deg}=k}
\end{align*}

\end{proposition}
\begin{proof}
The first one is just a restatement of proposition~\ref{Rstar1}.
The first two columns of \eqref{QDER} give us \[R(z)_{1\bar{\alpha}}=\frac{1}{I_{1,1}}(zq\frac{d}{dq}+L_\alpha(q))R(z)_{0\bar{\alpha}}\]
\[R(z)_{2\bar{\alpha}}=\frac{1}{I_{2,2}}(zq\frac{d}{dq}+L_\alpha(q))R(z)_{1\bar{\alpha}}\]
Then the recursion relation for $\{{R(z)_{i}}^{\alpha}\}$ is
\[{R(z)_{1}}^{\alpha}=\frac{1}{I_{1,1}}\left(\|e_\alpha\|^{-1}zq\frac{d}{dq}
\left(\|e_\alpha\|{R(z)_{0}}^{\alpha}\right)
+L_\alpha(q){R(z)_{0}}^{\alpha}\right)\]
\[{R(z)_{2}}^{\alpha}=\frac{1}{I_{2,2}}\left(\|e_\alpha\|^{-1}zq\frac{d}{dq}
\left(\|e_\alpha\|{R(z)_{1}}^{\alpha}\right)
+L_\alpha(q){R(z)_{1}}^{\alpha}\right)\]
By using equation \eqref{norm} and  equations $(i)$ in lemma~\ref{ZZYY}, we get
\begin{align}\label{Rrecursion}
&{R(z)_{1}}^{\alpha}=\frac{{I_{00}}^2}{L^2}\xi^\alpha\lambda {R(z)_0}^\alpha+z\frac{{I_{00}}^2}{L^2}\left(-X_{1} {R(z)_0}^\alpha+L^{-1}q\frac{d}{dq}{R(z)_0}^\alpha\right)
\\&{R(z)_{2}}^{\alpha}=\frac{I_{00}}{L}\left((\xi^\alpha\lambda)^2{R(z)_0}^\alpha
+2z\xi^\alpha\lambda\cdot L^{-1}q\frac{d}{dq}{R(z)_0}^\alpha+\frac{1}{9}z^2
\left((L^4-L){R(z)_0}^\alpha+9\left(L^{-1}q\frac{d}{dq}\right)^{2}{R(z)_0}^\alpha\right) \right)
\nonumber\end{align}
Notice that
\[L^{-1}q\frac{d}{dq}L=\frac{1}{3}L^3-\frac{1}{3}.\]
Combining with proposition~\ref{Rstar1},
we prove the proposition.
\end{proof}
\subsection{Finite generation property}
\begin{theorem}\label{finitgen}
\[\left[\Omega_{g,n}^{\tau(q)}(H^{i_1},...,H^{i_n})\right]_{d} \in \left(\frac{I_{00}}{L}\right)^{2g-2+\sum_{j=1}^{n}(2\delta_{i_{j}}^{1}+\delta_{i_j}^{2})}\lambda^{g-1+\sum_{j=1}^{n}i_j-d} \mathbb{Q}[X_1,L]_{\deg=d}\otimes A^{d}(\overline{\mathcal{M}}_{g,n};\mathbb{Q})
\]
where $[\quad]_{d}$ means the complex cohomological  degree $d$ part.
\end{theorem}
\begin{proof}
For simplicity, we first compute the contribution of the trivial stable graph
to ${\Omega^{\tau(q)}}_{g,n}(H^{i_1},...,H^{i_n})$:
\begin{align*}
&(T\omega)_{g,n}(R(z)^{-1}(H^{i_1}),..,R(z)^{-1}(H^{i_n}))
\\=&\sum_{k=0}^{\infty}\sum_{\alpha}{\Delta_\alpha}^{\frac{2g-2}{2}}\frac{1}{k!}(\pi_k)_*
\Big({R(-z)_{i_1}}^{\alpha}...{R(-z)_{i_n}}^{\alpha}
\\&\hspace{150pt}\psi_{n+1}\left[{\Psi_{0}}^{\alpha}-{R(-\psi_{n+1})_{0}}^{\alpha}\right]
...
\psi_{n+k}\left[{\Psi_{0}}^{\alpha}-{R(-\psi_{n+k})_{0}}^{\alpha}\right]\Big)
\end{align*}
Recall
\[{\Delta_{\alpha}}^{\frac{1}{2}}={\Psi_{0\bar\alpha}}^{-1}
=(\xi^\alpha\lambda)^{\frac{1}{2}}\frac{I_{0}}{L}\]
so
\[{\Delta_{\alpha}}^{\frac{2g-2}{2}}
=(\xi^\alpha\lambda)^{g-1}\left(\frac{I_{0}}{L}\right)^{2g-2}
\]
then the degree $d$ term equals to
\begin{align*}
&\left[(T\omega)_{g,n}(R(z)^{-1}(H^{i_1}),..,R(z)^{-1}(H^{i_n}))\right]_{\deg=d}
\\=&[z^d]\sum_{k=0}^{\infty}\sum_{\alpha}{\Delta_\alpha}^{\frac{2g-2}{2}}\frac{1}{k!}(\pi_k)_*
\Big({R(-z)_{i_1}}^{\alpha}...{R(-z)_{i_n}}^{\alpha}
\\&\hspace{150pt}\psi_{n+1}\left[{\Psi_{0}}^{\alpha}-{R(-z)_{0}}^{\alpha}\right]
...
\psi_{n+k}\left[{\Psi_{0}}^{\alpha}-{R(-z)_{0}}^{\alpha}\right]\Big)
\\ \in& \lambda^{g-1+\sum_{j=1}^{n}i_j-d} \left(\frac{I_{00}}{L}\right)^{2g-2+\sum_{j=1}^{n}
(2\delta_{i_{j}}^{1}+\delta_{i_j}^{2})} \mathbb{Q}[X_1,L]_{\deg=d}\otimes A^{d}(\overline{\mathcal{M}}_{g,n};\mathbb{Q})
\end{align*}
For general genus-$g$, $n$ marking  stable graph $\Gamma$, the associated contribution of $\Gamma$
 in Givental-Teleman graph sum formula  is
\begin{itemize}
\item{On each vertex, the contribution is
\[\omega_{g_v,n_v}(e_\alpha,...,e_\alpha)={\Delta_{\alpha}}^{\frac{2g_v-2}{2}}
=(\xi^\alpha\lambda)^{\frac{2g_v-2}{2}}\left(\frac{I_{0}}{L}\right)^{2g_v-2}\]}
\item{On the legs, the contribution is
\[\langle R(z)^{-1}H^i,e^\alpha\rangle={R(-z)_i}^\alpha\]}
\item{On the kappa tails, the contribution is
\[\psi({\Psi_0}^\alpha-{R(-\psi)_0}^\alpha)\]}
\item{On the edges, the contribution is
\begin{align*}
\langle V(z,w),e^\alpha\otimes e^\beta\rangle
=\sum_{i,j=0}^{2}\frac{\eta^{ij}}{z+w}
\left({\Psi_i}^\alpha\cdot{\Psi_j}^\beta-{R(-z)_i}^\alpha\cdot {R(-w)_j}^\beta\right)
\end{align*}
where $(\eta^{ij})$ is the inverse matrix of the paring matrix \ref{paring-matrix}. Notice that
\[[z^k w^l]\left(\sum_{i,j=0}^{2}\eta^{ij}{R(-z)_i}^{\alpha}{R(-w)_j}^{\beta}\right)
\in \lambda^{-k-l+1}(\xi^\alpha)^{-k+i}(\xi^\beta)^{-l+j}
\frac{{I_{00}}^2}{L^2}\cdot\mathbb{Q}[L]_{\deg=k+l}.\]
}
\end{itemize}

Assume the kappa tails contributes degree $t$ algebraic cohomological classes. Then we can analysis each factor in the contribution of $\Gamma$ to $\left[{\Omega^{\tau(q)}}_{g,n}\left(R(z)^{-1}(H^{i_1}),..,R(z)^{-1}(H^{i_n})
\right)\right]_{\deg=d}$:
The power of $\lambda$ is
\[\sum_{v}(g_v-1)+|E|+\sum_{j=1}^{n}i_j-(d-t)-t=g-1+\sum_{j=1}^{n}i_j-d.\]
The power of $\frac{I_{00}}{L}$ is
\[\sum_{v}(2g_v-2)+2|E|+\sum_{j=1}^{n}(2\delta_{i_j}^1+\delta^2_{i_j})
=2g-2+\sum_{j=1}^{n}(2\delta_{i_j}^1+\delta^2_{i_j}).\]
The rest coefficients lies in the the ring
\[\mathbb{Q}[L]_{\deg=t}\cdot \mathbb{Q}[X_1,L]_{\deg=d-t}\subset
\mathbb{Q}[X_1,L]_{\deg=d}.\]
Lastly, since the graph sum formula is symmetric to any index $\alpha$, so after taking summation, the factors $\xi^{\alpha},\xi^{\beta}...$ become some rational numbers.
\end{proof}
\section{Quasi-modularity property}\label{sec:modular}
\subsection{Review of quasi-modular forms of $\Gamma_0(3)$}
Let
\begin{align*}
  a(Q) &= \sum_{n, m \in \mathbb Z} Q^{n^2 + nm + m^2} \\
  b(Q) &= \sum_{n, m \in \mathbb Z} \omega^{n - m} Q^{n^2 + nm + m^2}
\end{align*}
be the cubic AGM theta functions (c.f. \cite{garvan1994cubic}). The ring of modular form and quasi-modular form of $\Gamma_0(3)$ is
\[\text{Mod}(\Gamma_0(3))=\mathbb{Q}[a(Q)^2,E_4(Q),E_6(Q)]\]
\[\text{QMod}(\Gamma_0(3))=\mathbb{Q}[a(Q)^2,E_2(Q),E_4(Q),E_6(Q)]\]
where $E_2,E_4,E_6$ are the Eisenstein series of weight 2,4,6 respectively.
\begin{remark}
By computing finitely many Fourier coefficients, we can check that  $a^2(Q)$, $a(Q) b^3(Q)$ and $b^6(Q)$ is another set of generators of the ring of modular form of $\Gamma_0(3)$ of weight 2,4 and 6 respectively.
\end{remark}

Since the ring of quasi-modular form is closed under the derivative operator $Q\frac{d}{dQ}$, then by comparing finitely many Fourier coefficient, we get the following identities.
\begin{lemma}\label{lm:derab}
\begin{align*}
  Q \frac{da}{dQ} &= \frac{1}{12} a(Q) E_2(Q) + \frac{1}{4} a(Q)^3 - \frac 13 b^3(Q) \\
  Q \frac{db^3}{dQ} &= \frac{1}{4} b(Q)^3 E_2(Q) - \frac{1}{4} a(Q)^2 b(Q)^3.
\end{align*}
\end{lemma}
\begin{proof}
If $f(Q)$ is a modular form of weight $k$ and level $n$ for some
character $\chi$, then the Serre derivative
\begin{equation*}
  Q \frac{df}{dQ} - \frac{k}{12} E_2(Q) f(Q)
\end{equation*}
is a modular form of weight $k + 2$ and level $n$ for the same
character.
Computing finitely many Fourier coefficient, we find the identities.
\end{proof}
\subsection{Relate generators to quasi-modular form}
Recall that
\begin{equation*}
  I_0(q) = \sum_{d = 0}^\infty q^d \frac{(3d)!}{(d!)^3} = {}_2 F_1\left(\frac 13, \frac 23; 1; 27q\right),
\end{equation*}
and
\begin{equation*}
  L(q) = (1 - 27q)^{-\frac 13}, \qquad I_{11}(q) = \frac{L^3}{I_0^2}.
\end{equation*}
On the other hand, the cubic AGM theta function
\begin{equation*}
  a(Q) = {}_2 F_1\left(\frac 13, \frac 23; 1; 1 - \frac{b^3(Q)}{a^3(Q)}\right).
\end{equation*}
Naturally we take the  identification
\begin{align}\label{identi}
  27q = 1 - \frac{b^3(Q)}{a^3(Q)}.
\end{align}
then
\begin{align}\label{aLmod}
a(Q) = I_0(q),\quad   L^{-3}(q) = \frac{b^3(Q)}{a^3(Q)}.
\end{align}
Moreover, we have the following lemma
\begin{lemma}
Under the identification \eqref{identi}, the variable $Q$ is exactly the mirror map
\[Q=q\exp\left(\frac{I_1(q)}{I_0(q)}\right).\]
\end{lemma}
\begin{proof}
By definition, $x = q e^{\frac{I_1}{I_0}}$ satisfies the differential equation
\begin{equation}
  \label{eq:DE}
  q \frac{dx}{dq} = I_{11} x = \frac{L^3}{I_0^2} x
\end{equation}
Using lemma~\ref{lm:derab}, the differential equation~\eqref{eq:DE} simplifies to
\begin{equation*}
  Q \frac{dx}{dQ} = x.
\end{equation*}
Since $Q$ has no constant term when written in terms of $x$, this
implies that $x = Q$.
\end{proof}
Now we can relate the basic generators to the quasi-modular from of $\Gamma_0(3)$.
\begin{lemma}\label{lem:mod-gene}
\begin{align*}
\frac{{I_{00}}^2}{L^2} X_{1}&=\frac{1}{12}E_2(Q)-\frac{1}{12}a(Q)^2\\
\frac{{I_{00}}^2}{L^2} Y_{1}&=\frac{1}{3}a(Q)^2\\
\left(\frac{{I_{00}}^2}{L^2}\right)^2Y_2&=-\frac{1}{36}a(Q)^4+\frac{1}{36}E_4(Q)\\
\left(\frac{{I_{00}}^2}{L^2}\right)^3Y_3&=\frac{1}{216}a(Q)^6
+\frac{1}{216}a(Q)^2E_4(Q)-\frac{1}{108}E_6(Q)
\end{align*}
\end{lemma}
\begin{proof}
By lemma~\ref{lm:derab}
\begin{align*}
  X_1 =\frac qL \frac d{dq} \log\left(\frac{I_0}L\right)
  = \frac{L^2}{3 I_0^2} Q \frac d{dQ} \log\left(b^3\right)
  = \frac 14 \frac{L^2}{I_0^2} E_2(Q^3) - \frac 14 L^2.
\end{align*}
The second and third ones follow from lemma\ref{lm:derab} and equations~\eqref{aLmod}.
\end{proof}
\begin{remark}
From lemma~\ref{lem:mod-gene}, we see that
\[
\frac{{I_{00}}^2}{L^2} Y_{1},\quad
\left(\frac{{I_{00}}^2}{L^2}\right)^2Y_2,\quad
\left(\frac{{I_{00}}^2}{L^2}\right)^3Y_3\]
is another set of generators of the ring of modular form of $\Gamma_0(3)$. And
\[\frac{{I_{00}}^2}{L^2} X_{1},\quad
\frac{{I_{00}}^2}{L^2} Y_{1},\quad
\left(\frac{{I_{00}}^2}{L^2}\right)^2Y_2,\quad
\left(\frac{{I_{00}}^2}{L^2}\right)^3Y_3\]
is another set of generators of the ring of quasi-modular form of $\Gamma_0(3)$.
\end{remark}
\subsection{Proof of quasi-modularity property of the twisted theory $\mathcal{O}(3)$ over $\mathbb{P}^2$}
First, we study the relations between basic generators and modular forms.
Taking $d=g-1+n$ in the theorem~\ref{finitgen}, combining the lemma~\ref{lem:mod-gene}, we obtain the following quasi-modularity property for twisted theory $\mathcal{O}(3)$ over $\mathbb{P}^2$.
\begin{corollary}
For $\lambda$-twisted theory $\mathcal{O}(3)$ over $\mathbb{P}^2$,
the degree $g-1+n$ part of ${\Omega^{\tau(q)}}_{g,n}(H,...,H)$ is a cycle-valued quasi modular form of $\Gamma_{0}(3)$ with weight $2g-2+2n$.
\[{\Omega^{\tau(q)}}_{g,n}(H,...,H)\cap A^{g-1+n}(\overline{\mathcal{M}}_{g,n};\mathbb{Q})\in \text{QMod}(\Gamma_0(3))_{2g-2+2n}\otimes_{\mathbb{Q}} A^{g-1+n}(\overline{\mathcal{M}}_{g,n};\mathbb{Q}).\]
\end{corollary}
\begin{remark}
For other degree parts, we can also express the coefficients  as the rational function of the generators of quasi-modular form of ring $\Gamma_0(3)$.
\end{remark}
\section{Holomorphic anomaly equation}\label{sec:HAE}

In this section, we prove the holomorphic anomaly equation. This part is similar as the  proof in \cite{lho2018stable} for the case of local $\mathbb{P}^2$. The key point is the following proposition:
\begin{proposition}\label{derRV}
For $R$ matrix,
\begin{align}\label{DerR}
\frac{\partial}{\partial E_{2}} R(z)^{-1}H^{i}
=\frac{1}{12}\delta_{i,1}z R(z)^{-1}1
\end{align}
For $V$ matrix,
\begin{align}\label{DerV}
\frac{\partial}{\partial E_2}V(z,w)=-\frac{1}{36}R(z)^{-1}1\otimes R(w)^{-1}1
\end{align}
\end{proposition}
\begin{proof}
Equation~\eqref{DerR} follows from equation~\eqref{Rrecursion} and lemma~\ref{lem:mod-gene}. Equation~\eqref{DerV} just follows from equation~\eqref{DerR}.
\end{proof}
Using proposition~\ref{derRV} and Givental-Teleman classification theorem, we obtain
\begin{theorem}For the formal twisted theory $\mathcal{O}(3)$ over $\mathbb{P}^2$, the holomorphic anomaly equation holds
\begin{align*}
&\frac{\partial}{\partial E_{2}}{\Omega^{\tau(q)}}_{g,n}(H^{i_1},...,H^{i_n})
\\=&\frac{1}{12}\sum_{j=1}^{n}\delta_{i_j}^{1}\psi_{j}
{\Omega^{\tau(q)}}_{g,n}(H^{i_1},...,H^{i_{j-1}},1,H^{i_{j+1}}...,
H^{i_n})
\\&-\frac{1}{36}
i_{*}\left({\Omega^{\tau(q)}}_{g-1,n+2}(H^{i_1},...,H^{i_n},
1,1)
\right)
\\&-\frac{1}{36}
j_{*}\sum_{\substack{g_1+g_2=g\\S_1\sqcup S_2=\{1,...,n\}}}{\Omega^{\tau(q)}}_{g_1,|S_1|+1}
(H^{S_1},1)\otimes
{\Omega^{\tau(q)}}_{g_2,|S_2|+1}(H^{S_2},1)
\end{align*}
\end{theorem}
\begin{proof}
By Teleman-Givental theorem, ${\Omega^{\tau(q)}}_{g,n}(H^{i_1},...,H^{i_n})$ can be expressed as a graph sum formula. Taking derivative along $\frac{\partial}{\partial E_2(Q)}$, using Leibniz's rule, together with proposition~\ref{derRV}, we prove the theorem.
\end{proof}

\bibliographystyle{plain}

\begin{thebibliography}{10}

\bibitem{coates2009computing}
Tom Coates, Alessio Corti, Hiroshi Iritani, Hsian-Hua Tseng, et~al.
\newblock Computing genus-zero twisted gromov-witten invariants.
\newblock {\em Duke Mathematical Journal}, 147(3):377--438, 2009.

\bibitem{coates2007quantum}
Tom Coates and Alexander Givental.
\newblock Quantum riemann-roch, lefschetz and serre.
\newblock {\em Annals of mathematics}, pages 15--53, 2007.

\bibitem{coates2018gromov}
Tom Coates  and  Hiroshi Iritani.
\newblock Gromov-Witten Invariants of Local $\mathbb{P}^2$ and Modular Forms.
\newblock {\em arXiv preprint arXiv:1804.03292}, 2018.

\bibitem{fang2018open}
Bohan Fang, Yongbin Ruan, Yingchun Zhang  and  Jie Zhou.
\newblock Open Gromov-Witten Theory of $ K_{\mathbb{P}^2}, K_{\mathbb{P}^1\times\mathbb{P}^1}, K_{\mathbb{WP}[1,1,2]}, K_{\mathbb{F}_1}$ and Jacobi Forms.
\newblock {\em arXiv preprint arXiv:1805.04894}, 2018.


\bibitem{garvan1994cubic}
Frank Garvan.
\newblock Cubic modular identities of ramanujan, hypergeometric functions and
  analogues.
\newblock In {\em The Rademacher Legacy to Mathematics: The Centenary
  Conference in Honor of Hans Rademacher, July 21-25, 1992, the Pennsylvania
  State University}, volume~1, page 245. American Mathematical Soc., 1994.

\bibitem{givental1996equivariant}
Alexander~B Givental.
\newblock Equivariant gromov-witten invariants.
\newblock {\em International Mathematics Research Notices}, 1996(13):613--663,
  1996.

\bibitem{givental2001semisimple}
Alexander~B Givental.
\newblock Semisimple frobenius structures at higher genus.
\newblock {\em International mathematics research notices},
  2001(23):1265--1286, 2001.

\bibitem{givental2004symplectic}
Alexander~B Givental.
\newblock Symplectic geometry of frobenius structures.
\newblock In {\em Frobenius manifolds}, pages 91--112. Springer, 2004.
\bibitem{Guo2019}
Shuai Guo, Felix Janda and Yongbin Ruan.
\newblock The structures of higher genus Gromov-Witten inavariants of quintic 3-folds.
\newblock {\em arXiv preprint arXiv:1812.11908}.
\bibitem{lho2018stable}
Hyenho Lho and Rahul Pandharipande.
\newblock Stable quotients and the holomorphic anomaly equation.
\newblock {\em Advances in Mathematics}, 332:349--402, 2018.

\bibitem{lho2018crepant}
Hyenho Lho and Rahul Pandharipande.
\newblock Crepant resolution and the holomorphic anomaly equation for $\mathbb{C}^3/\mathbb{Z}_3$.
\newblock {\em arXiv preprint arXiv:1804.03168}, 2018.

\bibitem{li1998virtual}
Jun Li and Gang Tian.
\newblock Virtual moduli cycles and gromov-witten invariants of algebraic
  varieties.
\newblock {\em Journal of the American Mathematical Society}, 11(1):119--174,
  1998.

\bibitem{lian1997mirror}
Bong~H Lian, Kefeng Liu, and Shing-Tung Yau.
\newblock Mirror principle I.
\newblock {\em Asian Journal of Mathematics}, 1(4):729--763, 1997.

\bibitem{oberdieck2018holomorphic}
Georg Oberdieck and Aaron Pixton.
\newblock Holomorphic anomaly equations and the Igusa cusp form conjecture.
\newblock {\em Inventiones mathematicae}, 213(2):507--587, 2018.


\bibitem{pandharipande2015relations}
Rahul Pandharipande, Aaron Pixton, and Dimitri Zvonkine.
\newblock Relations on $\overline{M}_{g,n}$ via 3-spin structures.
\newblock {\em Journal of the American Mathematical Society}, 28(1):279--309,
  2015.

\bibitem{ruan1997higher}
Yongbin Ruan and Gang Tian.
\newblock Higher genus symplectic invariants and sigma models coupled with
  gravity.
\newblock {\em Inventiones mathematicae}, 130(3):455--516, 1997.

\bibitem{teleman2012structure}
Constantin Teleman.
\newblock The structure of 2d semi-simple field theories.
\newblock {\em Inventiones mathematicae}, 188(3):525--588, 2012.

\bibitem{zagier2007some}
Don Zagier and Aleksey Zinger.
\newblock Some properties of hypergeometric series associated with mirror
  symmetry.
\newblock {\em arXiv preprint arXiv:0710.0889}, 2007.

\end{thebibliography}

\end{document}